\documentclass[11pt,leqno]{amsart}

\usepackage[utf8x]{inputenc}          
\usepackage{lmodern}                
\usepackage[T1]{fontenc}            
\usepackage[english]{babel}         
\usepackage{ucs}                      

\usepackage{geometry}
\usepackage{graphicx}     
\usepackage{xcolor}       
\usepackage{pstricks-add} 
\usepackage{enumitem}     
\usepackage{verbatim}     

\usepackage{amsmath, amsthm}
\usepackage{amsfonts, amssymb}
\usepackage{bm}               

\usepackage{hyperref}

\theoremstyle{plain}
  \newtheorem{theorem}{Theorem}[section]
  \newtheorem*{theorem*}{Theorem}
  
  \newtheorem{proposition}{Proposition}[section]
  \newtheorem*{proposition*}{Proposition}

  \numberwithin{equation}{section}

\theoremstyle{remark}
  \newtheorem{remark}{\textbf{Remark}}[]

  \newtheorem*{c-example}{Counter-example}

\theoremstyle{definition}

\newcommand{\ve}{\varepsilon}
\newcommand{\eg}{\emph{e.g. }}
\newcommand{\ie}{\emph{i.e. }}
\newcommand{\RR}{\mathbb{R}}
\newcommand{\NN}{\mathbb{N}}
\newcommand{\ZZ}{\mathbb{Z}}
\newcommand{\ff}{\hat f}

\newcommand{\be}{\begin{equation}}
\newcommand{\ee}{\end{equation}}
\newcommand{\bd}{\begin{displaymath}}
\newcommand{\ed}{\end{displaymath}}
\newcommand{\ba}{\begin{eqnarray}}
\newcommand{\ea}{\end{eqnarray}}
\newcommand{\fer}[1]{(\ref{#1})}
\author[Th. Rey]{Thomas Rey}
\address{
    Thomas Rey \\
    The University of Maryland \\
    CSCAMM \\
    4146 CSIC Building \\
    Paint Branch Drive \\
    College Park, MD 20740 \\
    USA
}
\email{trey@cscamm.umd.edu}

\author[G. Toscani]{Giuseppe Toscani}
\address{
    Giuseppe Toscani \\
    Dipartimento di Matematica \\
    Universit\`a di Pavia \\
    via Ferrata 1 \\
    27100 Pavia, Italy.
}
\email{giuseppe.toscani@unipv.it}

\title[Rosenau-Type Approximations to the Heat Equation]
      {Large-time Behavior of the Solutions to Rosenau \\
       Type Approximations to the Heat Equation
       }


\date{}

\begin{document}

    \begin{abstract}
      In this article we study the large-time behavior of the solution to a general Rosenau type approximation to the heat equation \cite{Rosenau:1992}, by showing that the solution to this approximation approaches the fundamental solution of the heat equation at a sub-optimal rate.
      The result is valid in particular for the central differences scheme approximation of the heat equation, a property which to our knowledge has never been observed before.
    \end{abstract}

  \maketitle

  \tableofcontents

  \section{Introduction}

    \subsection{{Rosenau Regularization of the Chapman-Enskog Expansion}}

            {In his seminal work} \cite{Rosenau:1992}, Rosenau has proposed a regularized version of the Chapman-Enskog expansion of hydrodynamics. This regularized expansion resembles the usual Navier-Stokes viscosity terms at low wave-numbers, but unlike the latter, it has the advantage of being a bounded macroscopic approximation to the linearized collision operator.
            The model is given by the scalar equation
            \be\label{Rose}
              \frac{\partial f}{\partial t} +  \frac{\partial \Psi(f)}{\partial v} = \left[ \frac {- \ve\xi^2}{1 + \ve^2 m^2 \xi^2}\ff(\xi) \right]^{\vee},
            \ee
            where $\ff(\xi)$ denotes the Fourier transform of $f(v)$, while $ f(\xi)^\vee $ denotes the inverse Fourier transform.
			
            The operator on the right hand side looks like the usual viscosity term $\ve f_{vv}$ at low wave-numbers $\xi$, while for higher wave numbers it is intended to model a bounded approximation of a linearized collision operator, thereby avoiding the artificial instabilities that occur when the Chapman-Enskog expansion for such an operator is truncated after a finite number of terms.
			
            One of the advantages of this regularization is that it is equivalent in small frequencies to the heat equation (and then regularizes initial data), whereas high frequencies behave as an absorption term. In particular, information will travel at finite speed, and tools from hyperbolic equations can be used.
            Note that the right side of \fer{Rose} can be written as
            \be\label{Ros2}
              \left[ \frac {- \ve \xi^2}{1 + \ve^2m^2 \xi^2}\ff(\xi) \right]^{\vee} = \frac \ve{(\ve m)^2 } \left[ \frac 1 {1 +\ve^2 m^2 \xi^2}\ff(\xi) -\ff(\xi) \right]^{\vee} = \frac 1{m \bar\ve} \left[ M_{\bar\ve}*f -f\right],
            \ee
            where $\bar\ve = m\ve$,  $*$ denotes convolution and
            \be
              \label{Max}
              M_\gamma(v) = \frac 1{2\gamma}\,e^{-|v|/\gamma}
            \ee
            is a non-negative function satisfying $\| M_\gamma\|_{L^1} = 1$.
            In other words, the Rosenau approximation consists in substituting  the linear diffusion equation
            \be
              \label{heat}
              \frac{\partial g}{\partial t} (v,t) = \sigma^2 \frac{\partial^2 g}{\partial v^2}(v,t)
            \ee
            with the linear kinetic equation
            \be
              \label{kin}
              \frac{\partial g}{\partial t} (v,t) = \frac {\sigma^2}{\ve^2} \left[ M_{\ve}*g(v,t) -g(v,t)\right]
            \ee
            in which the ``Maxwellian'' $M_\ve$  \cite{Cercignani:1988} is given by \fer{Max}.
			
            Equation \fer{Rose} has then been studied by S. Schochet and E. Tadmor in \cite{Schochet:1992} in the context of vanishing viscosity solutions of scalar conservation laws.
            {H. Liu and E. Tadmor then studied in \cite{Liu:2001} the Burgers equation with a Rosenau-like nonlocal viscosity term and proved the existence of a so-called \emph{critical threshold}, namely the existence of a critical value of the total variation of the initial condition triggering or not a finite-time blow-up of the solutions.
            This  nonlocal term} was also used by C. Rohde to model capillarity effects close to fluid-vapour phase transitions in Navier-Stokes equations \cite{Rohde:2005}. Interestingly, this non-local approach seems more physically relevant than the original Korteweg's  idea to use the heat operator for describing such phenomenons.
			
            More recently, it has been replaced in a general context of system of hyperbolic balance laws with non-local source term by R. Colombo and G. Guerra in the two companion papers \cite{colombo:2007,colombo:2008}.
            They showed in particular the well posedness in $L^1$ (globally in time) of the Cauchy problem for~\eqref{Rose} together with uniform stability estimates, in the framework of viscosity solutions.
            {Finally, C. Besse and T. Goudon obtained recently in \cite{Besse:2010} a system of macroscopic equations with a nonlocal Rosenau term by computing the diffusion limit of a linear space inhomogeneous kinetic equation.
            A similar approach was also used by T. Goudon and M. Parisot in \cite{Goudon:2011} on a kinetic equation describing the behavior of a multi-species gas of charged particles to obtain a macroscopic model with a Rosenau-like nonlocal term.
            }
			
            Despite the previous studies, it is not completely clear and well established from a mathematical point of view if the correction proposed by Rosenau is a \emph{good} approximation to the linear heat equation. In particular, it is not clear whether or not the large-time behavior of the solution to the Rosenau approximation agrees with the large-time behavior of the linear diffusion equation.
			
            The aim of this article is to give an answer to the previous question, and to underline that the Rosenau approximation can be viewed as a particular case of a general approximation to the heat equation by means of a linear kinetic equation of type \fer{kin}, provided the \emph{background density} $M_\ve$ is a probability density function of zero mean and variance proportional to $\ve^2$.
            In particular, it will be shown that, in a certain metric equivalent to the weak*-convergence of measures, the distance between the solution to the heat equation and the solution to the kinetic equation can be bounded uniformly in terms of $\ve$ and $t$, provided the background density has a sufficiently high number of moments (typically more than two).
			
            The plan on the article is as follows. In Section \ref{secKinetic}, using tools of the kinetic theory of rarefied gases, we will introduce a possible derivation and the main features of the Rosenau approximation with a general kernel. This allows to describe well known models such as the central differences schemes for the heat equation. In Section \ref{secRepr},  this kinetic formulation will be used to obtain explicit solutions to the Rosenau equation \eqref{kin}, using both Fourier transform and Wild sums. Last, in Section \ref{secAsympt}, we will investigate the large time behavior of the solutions to \eqref{kin}, showing that the convergence towards the fundamental solution occurs, in a suitable Fourier-based metric ({see next subsection}), at a suboptimal rate (compared to the heat equation). Finally, we will combine these results to show that strong convergence in $L^1$ towards the fundamental solution to the heat equation is obtained after a suitable regularization of the Rosenau equation, obtained by discarding its singular part.

    \subsection{{Functional Framework}}
      \label{subFuncFram}

      Before entering into the main topic of this paper, we list below the various functional spaces used in the following.
      For $p \in [1, +\infty)$ and $q \in [1, +\infty)$, we denote by $L_q^p$  the weighted Lebesgue spaces
            \begin{equation*}
                L^p_q := \left\{ f : \mathbb{R} \rightarrow \mathbb{R} \text{ measurable; }\|f\|_{L^p_q}^p := \int_{\mathbb{R}} |f(v)|^p \, (1+ v^2)^{q/2} \, dv < \infty \right\}.
            \end{equation*}
             In particular, the usual Lebesgue spaces are given by
            \[ L^p := L^p_0.\]
            Moreover, for $f \in L^1_q$, we can define for any $k \leq q$  the $k^{th}$ order \emph{moment} of $f$ as the quantity
            \begin{equation*}
              m_k(f):= \int_\RR f(v) \, |v|^{k} dv \, < \, \infty.
            \end{equation*}
            For $s \in \NN$, we denote by $W^{s,p}$ the Sobolev spaces
            \begin{equation*}
                W^{s,p} := \left\{ f \in L^s; \|f\|_{W{s,p}}^p := \sum_{|k| \leq s} \int_{\RR} \left |f^{(k)}(v)\right |^p \, dv < \infty \right \}.
            \end{equation*}
          If $p=2$ we set $H^s := W^{s,2}$.

          Given a probability density $f$, we define its \emph{Fourier transform} $\mathcal F_v(f)$ by
        \begin{equation*}
          \mathcal F_v(f)(\xi) = \widehat{f}(\xi) := \int_\RR e^{- i \, \xi \, v} f(v)\, dv, \qquad \forall  \xi \in \RR.
        \end{equation*}
        The Sobolev space $H^s$ can equivalently be defined for any $s \geq 0$ by the norm
            \begin{equation*}
                \|f\|_{H^{s}} := \left \| \mathcal F_v \left (f \, \right )\right \|_{L^2_{2s}}.
            \end{equation*}
            The homogeneous Sobolev space $\dot H^s$ is then defined by the homogeneous norm
            \begin{equation*}
              \|f\|_{\dot H^s}^2 := \int_\RR |\xi|^{2 s} \left | \widehat{f}(\xi) \right |^2 \, d \xi.
            \end{equation*}

                  Finally, we introduce a family of Fourier based metrics in the following way: given $s >0$ and two probability distributions $f_1$ and $f_2$, their Fourier based distance $d_s(f_1,f_2)$ is the quantity
            \begin{equation*}
            d_s(f_1,f_2) := \sup_{\xi \in \RR \setminus 0} \frac{\left |\widehat{f_1}(\xi)-\widehat{f_2}(\xi)\right |}{|\xi|^s}.
            \end{equation*}
            This distance is finite, provided that $f_1$ and $f_2$ have the same moments up to order $[s]$, where, if $s \notin \NN$,  $[s]$ denotes the entire part of $s$, or up to order $s-1$ if $s \in \NN$.
            Moreover $d_s$ is an \emph{ideal} metric \cite{Carrillo:2007}.
            Its main properties are the following
            \smallskip
            \begin{enumerate}[leftmargin=*]
                \item For all probability distributions $f_1$, $f_2$, $f_3$,
                \[ d_s(f_1* f_3, f_2*f_3) \leq d_s (f_1, f_2); \]
                \item Define for a given nonnegative constant $a$ the dilatation \[f_a(v) = \frac 1a f\left ( \, \frac va \, \right ).\] Then for all probability distributions $f_1$, $f_2$, and any nonnegative constant $a$
                \[ d_s( f_{1,a}, f_{2, a}) = a^s \, d_s(f_1, f_2). \]
            The $d_s$-metric is related to other more known metrics of large use in probability theory \cite{Gabetta:1995}.
            \end{enumerate}

\section{A Kinetic Description of the Rosenau Approximation}
  \label{secKinetic}

  Let $v \in \RR$ denote velocity, and let us assume to have at time $\tau>0$  a system of particles immersed in a background. Let us suppose that the number of particles in the system is sufficiently large to be studied by means of statistical mechanics, namely by giving the velocity distribution $f(v,\tau)$ at time $\tau >0$.
  Moreover, let us assume that the main phenomenon which can modify particle's velocity is the interaction of particles with the background. Let $M(w)$ denote  the fixed in time (probability) distribution of the particles of the background, which we will assume of finite variance $\gamma^2$.
Assume that the collision process of a particle with velocity $v$ with a background particle with a velocity $w$  generates a post-collision velocity $v^*$  given by
\begin{equation}
  \label{collProc}
  v^* = v + w .
\end{equation}
Then, in a suitable scaling \cite{Cercignani:1988}, the effect of interactions \fer{collProc} on the time-variation of the density $f(v,\tau)$  can be quantitatively described by a linear Boltzmann-type equation, in which the variation of the density is due to  a balance between a gain and loss terms.

More precisely, for a given number $v$, we take into account all the interactions of type \fer{collProc} which end up with the number $v^*$ (gain term) as well as all the interactions
which, starting from the number $v$, lose this value after interaction (loss term).
The balance equation for the density of particles can be fruitfully written in weak form. It corresponds to say that the aforementioned interaction process on particles modifies the solution $f(v,\tau)$ according to
\be
  \label{kine-w}
  \frac{d }{d\tau}\int_{\RR} \varphi(v) f(v,\tau)\,dv =\lambda\, \int_{\RR^2}
  \bigl( \varphi(v^*)-\varphi(v)\bigr) f(v,\tau) M(w) \,dv \, dw ,
\ee
where the constant $\lambda> 0$ denotes the intensity of the variation process, and $\varphi(v)$ is a smooth function.
Note that choosing $\varphi(v) = 1$ shows that, independently of the background distribution, $f(v,\tau)$ remains a probability
density if it so initially
\begin{equation*}
  \int_{\RR} f(v, \tau)\,dv = \int_{\RR} f_0(v)\,dv = 1 .
\end{equation*}
This is in general the unique conservation law associated to equation \fer{kine-w}.

From now on, let us assume in addition that the probability distribution of the background is centered, and its variance depends on a small parameter $\ve >0$. To emphasize this dependence, we will denote this distribution by  $M_\ve(w)= \ve^{-1} M(\ve^{-1} w)$.  Then $M_\ve(w)$ satisfies
 \be\label{norm}
  \int_\RR  M_\ve(w) \, dw = 1, \quad\int_\RR w \, M_\ve(w) \, dw = 0, \quad  \int_\RR w^2 M_\ve(w) \, dw  = \ve^2 \,\gamma^2.
 \ee
 The weak formulation \fer{kine-w} yields immediately the time evolution  of the moments of $f$. Taking $\varphi(v) = v$ one obtains
\begin{equation*}
\frac{d}{d\tau} \int_{\RR} v f(v,\tau) \, dv  = \lambda  \int_\RR w \, M_\ve(w) \, dw \int_\RR f(v,\tau) \, dv =  0.
\end{equation*}
 Moreover, if $\varphi(v) = v^2$
\be
\frac{d}{d\tau} \int_{\RR} v^2 f(v,\tau) \, dv = \lambda \, \int_{\RR^2} \left[ (v^*)^2 - v^2 \right] f(v,\tau) \, M_\ve(w) \, dv \, dw  = \lambda \, \ve^2  \,\gamma^2. \label{eqKinEvol}
\ee
Thus, the second moment of $f$ grows linearly with respect to time and depends on $\ve$. One way to avoid this dependency is to scale the time properly. Setting $t = \ve^2 \tau$ and introducing a new particles distribution function $g_\ve$ such that $g_\ve(v, t) = f(v, \tau)$ gives according to \eqref{eqKinEvol},
\begin{equation*}
\frac{d}{dt} \int_{\RR} {v^2} g_\ve(v, t) \, dv = \frac{1}{\ve^2}\frac{d}{d\tau} \int_{\RR}{v^2} f(v, \tau) \, dv = \lambda \,\gamma^2,
\end{equation*}
and the second moment of $g_\ve$ does not depend on $\ve$.

The distribution $g_\ve$ is a weak solution to
\be
\frac{d}{dt} \int_{\RR} g_\ve(v, t) \varphi(v) \, dv = \frac{\lambda}{\ve^2} \int_{\RR^2} \left [ \varphi(v+w) - \varphi(v) \right ] g_\ve(v,t \,)M_\ve(w) \, dv \, dw, \label{eqBoltzResc}
 \ee
with $g_\ve(v,t=0) = g_0(v) =f_0(v)$.
Note that, since
\begin{align*}
\int_{\RR^2} \varphi(v +w ) g_\ve(v,t)\, M_\ve(w) \, dv \, dw & \, = \, \int_{\RR^2} \varphi(z) \, g_\ve(v,t)\, M_\ve(z-v) \, dv \, dz \\
 & \, = \, \int_{\RR} \varphi(z) M_\ve*g_\ve(z)\, dz,
\end{align*}
equation \fer{eqBoltzResc} can be rewritten as
 \be\label{Ros}
\frac{d}{dt} \int_{\RR} g_\ve(v, t) \, \varphi(v) \, dv = \frac{\lambda}{\ve^2} \int_{\RR} \varphi(v)\left( M_\ve* g_\ve(v)- g_\ve(v)\right) \, dv,
 \ee
which is the weak form of \eqref{Ros2}.  Hence, if the background distribution is given by \fer{Max}, equation \fer{eqBoltzResc} coincides with Rosenau's approximation \fer{kin},  where $\sigma^2 = \lambda$. Of course, other choices of the background are possible, and, provided conditions \fer{norm} are satisfied, the evolution of the moments of the solution to \fer{eqBoltzResc} up to the second order do not depend on the background distribution.

It is then easy to see that, in the case in which $M_\ve$ coincides with the ``Maxwellian'' \fer{Max},  $g_\ve$ is an approximate (at the leading order in $\ve$) weak solution of the heat equation \eqref{heat}. Indeed, given that $\ve$ is small enough, one can Taylor expand $\varphi(v^*)$, where $v^* = v + w$, to obtain
\begin{equation*}
\varphi(v^*) = \varphi(v) + w  \,\varphi'(v) + \frac{w^2}{2} \, \varphi''(v) + \frac 1{3!}\varphi^{(3)}(\tilde v) w^3
\end{equation*}
for $\tilde{v} \in (v, v+w)$.
Using this relation in \eqref{eqBoltzResc}, one obtains
\begin{equation} \label{eqHeatWeak}
\frac{d}{dt} \int_{\RR} g_\ve(v, t) \varphi(v) \, dv = \frac{\lambda\gamma^2}2 \int_\RR g_\ve(v, t) \varphi''(v) \, dv + R(\ve)
\end{equation}
where the remainder $R(\ve)$ satisfies
 \be\label{rest}
|R(\ve)| \le  \frac \lambda{3!}\|\varphi^{(3)}\|_{L^\infty}\frac 1{\ve^2}  \int_{\RR} w^3 M_\ve(w) \, dw = \ve \, \frac \lambda{3!}\|\varphi^{(3)}\|_{L^\infty} \int_{\RR} w^3 M(w) \, dw.
 \ee
\begin{remark}\label{approx}
One can easily notice thanks to \eqref{eqHeatWeak} that equation \fer{eqBoltzResc}  is an approximation to the heat equation \fer{heat} (with diffusion coefficient $\sigma^2 = \lambda\, \gamma^2/2$), provided that the remainder converges to zero as $\ve \to 0$. In order to have this convergence, it is enough that the distribution of the background is such that some moment of order greater than two remains bounded. We will use this in the following to obtain various approximations of the heat equation, just  changing the distribution of the background.
\end{remark}

\begin{remark}\label{entropy}
The Rosenau type kinetic equation \fer{eqBoltzResc} is such that mass, momentum and energy  of its solution have the same evolution of the corresponding moments of the solution to the heat equation \fer{heat}. A further interesting analogy with the heat equation is given by studying the evolution of convex functionals along the solution.
Let $\Phi(r)$, $ r \ge 0$ be a (regular) convex function of $r$. Then, using equation \fer{Ros} we obtain
 \begin{align*}
  \frac{d}{dt} \int_{\RR} \Phi(g_\ve(v, t)) \, dv & =  \int_{\RR} \Phi' (g_\ve(v, t)) \frac{\partial g_\ve(v, t)}{\partial t}  \, dv \\
   & = \frac{\lambda}{\ve^2} \int_{\RR} \Phi' (g_\ve(v, t)) \left( M_\ve*g_\ve(v)- g_\ve(v)\right) \, dv .
 \end{align*}
Thanks to the convexity of $\Phi(\cdot)$,  for $r,s \ge 0$
 \[
\Phi'(s) (r-s) \le \Phi(r) - \Phi(s),
 \]
and one obtains
\[
\frac{d}{dt} \int_{\RR} \Phi(g_\ve(v, t)) \, dv \le \frac{\lambda}{\ve^2} \int_{\RR} \left( \Phi( M_\ve*g_\ve(v)) - \Phi( g_\ve(v)) \right) \, dv.
 \]
Now, use the fact that $M_\ve$ is a probability distribution, so that  by Jensen's inequality
\begin{align*}
  \int_{\RR}\Phi( M_\ve*g_\ve(v)) \, dv & =  \int_{\RR} \Phi\left(\int_{\RR}g_\ve(v-w))M_\ve(w)\, dw \right)\,dv \\
  & \le \int_{\RR^2} \Phi(g_\ve(v-w))M_\ve(w)\, dv \,dw = \int_{\RR} \Phi(g_\ve(v))\, dv,
\end{align*}
 and one concludes with
\[
\frac{d}{dt} \int_{\RR} \Phi(g_\ve(v, t)) \, dv \le 0.
 \]
Thus any convex functional is non-increasing along the solution to the Rosenau type kinetic equation \fer{kin}, in agreement with the analogous property of the heat equation.
\end{remark}

\begin{remark}\label{cds}
A leading example of Rosenau type approximation is obtained by assuming that the background distribution $M_\ve$ is a balanced Bernoulli distribution
\begin{equation} \label{CDKernel}
M_\ve(v) = \frac{1}{2}\left  [ \, \delta_0(v+ \ve \gamma) + \delta_0(v-\ve \gamma) \right ],
\end{equation}
where as usual $\delta_0$ denotes a Dirac mass concentrated at $v = 0$. Note that the ``Maxwellian'' \fer{CDKernel} satisfies \fer{norm}. Then, according to \eqref{Ros}, the distribution $g_\ve$ solves
\begin{equation} \label{centralDiff}
\frac{\partial g_\ve}{\partial t}(v, t) = \frac{\lambda}{2 \ve^2} \left [ \, g_\ve(v - \ve \gamma, t) - 2 g_\ve(v, t) + g_\ve(v + \ve \gamma, t) \right ].
\end{equation}

Now, let us fix $\lambda = 2$ and $\gamma = \sigma$. Then, given a parameter $\Delta v > 0$ and a uniform grid $v_i = i \Delta v, \,\,i \in \ZZ $ on $\RR$, by setting $g_i(t) := g_\ve(v_i, t)$ and $\ve = \Delta v / \sigma$,  equation \eqref{centralDiff}  reads
\begin{equation*}
g_i'(t) = \sigma^2 \frac{g_{i+1}(t) - 2 g_i(t) + g_{i-1}(t)}{\Delta v^2}.
\end{equation*}
This is exactly the classical semi-explicit second order central differences scheme for the heat equation \fer{heat}.
It follows that this scheme furnishes an approximation to the heat equation which satisfies all the properties outlined in Remarks \ref{approx} and \ref{entropy}, as well as the {weak convergence} towards the exact solution, which follows easily by applying the estimate \eqref{rest}.
\end{remark}

\begin{remark}
The previous example introduces into the matter a background with a non regular distribution (presence of point masses). This situation clearly differs from the standard Rosenau approximation, characterized by the regular density \fer{Max}.
Due to its wide applications in numerical simulations, the study of the properties of the Rosenau approximation given by this type of background turns out to be important. For this reason, in addition to the study of the classical Rosenau model, in what follows we will deal also with the Bernoulli type background, outlining when possible the principal differences.
\end{remark}

\section{Representations of the Solutions to the Rosenau Equation}
\label{secRepr}

The Rosenau approximation~\eqref{kin} of the heat equation is a linear kinetic equation of Boltzmann type. Therefore, we can resort both to linear kinetic theory and to the theory of linear diffusion equations to obtain explicit representations of the solution.
We shall present here two equivalent ways to construct these solutions, each one giving a different insight on the behavior of $g_\ve$.
In the rest of this Section, to avoid inessential difficulties, we will assume that the initial value $g_0(v)$ is a probability density with finite moments up to a given order (in general more than two).

 \subsection{Wild Sums}

  The first method we introduce to obtain an explicit representation of the solution  $g_\ve$ is closely related to  the so-called theory of \emph{Wild sum expansions} of the Boltzmann equation.
  Wild indeed proved in~\cite{Wild:1951} that one may represent solutions to the nonlinear Boltzmann equation for Maxwell molecules by using  convergent power series.
  This method is particularly simple when one deals with linear convolution equations such as~\eqref{kin}.
  To start with, let us set
  \[
  h_\ve(v,t) := \exp\left ( \frac{\lambda \, t}{ \ve^2 }\right ) g_\ve(v,t).
  \]
  Then the Cauchy problem for~\eqref{kin} can be rewritten as a fixed point problem (which is nothing but a Duhamel formula). Let $h \to \Phi_\ve(h)$ define the map
  \be \label{defFixPnt}
    \Phi_\ve(h) =  g_0 + \frac{\lambda }{ \ve^2 }\int_0^t  M_\ve * h(s) \, ds.
  \ee
 Then $g_\ve(v,t)$ solves~\eqref{kin} exactly when $\Phi_\ve(h_\ve) = h_\ve$.

  To find the fixed point, it is then sufficient to make a Picard iteration. Starting from $h^{(0)}_\ve :=g_0$ one defines, for $n \in \NN$
  \begin{equation} \label{defIteration}
    h^{(n+1)}_\ve = \Phi_\ve \left (h^{(n)}_\ve\right ).
  \end{equation}
 By recurrence, for $n \ge 1$ it yields
  \begin{equation}
        h^{(n)}_\ve(v,t) \, =  \ h^{(n-1)}_\ve(v,t) + \left ( \frac{\lambda \, t}{ \ve^2 }\right )^n \frac{1}{n!} \, M_\ve^{*_n} * g_0(v), \label{eqHN}
  \end{equation}
  where we use the shorthand $M^{*_n} := M * \cdots * M$ ($n$ times). Clearly, for $n \ge 0$
  \[ h^{(n+1)}_\ve(v,t) - h^{(n)}_\ve(v,t) \geq 0. \]
  Hence the sequence $\left (h^{(n)}_\ve(v,t)\right )_{n \geq 0}$, {being bounded in $L^1$}, converges towards $h_\ve(v,t) \geq 0$ when $n \to \infty$, where according to \eqref{eqHN}
  \begin{equation*}
    h_\ve(v,t) \, = \, g_0(v) + \sum_{n \geq 1} \left ( \frac{\lambda \, t}{ \ve^2 }\right )^n \frac{1}{n!} \, M_\ve^{*_n} * g_0(v).
  \end{equation*}
  By passing to the limit $n \to \infty$ in \eqref{defIteration}, one obtains that $h_\ve$ is a (nonnegative) fixed point for \eqref{defFixPnt}. This procedure allows to write the (nonnegative) solution $g_\ve$ to the Cauchy problem for \eqref{kin} as
  \begin{equation} \label{eqGpowSerie}
     g_\ve(v,t) \, = \, e^{-\lambda \, t/\ve^2} g_0(v) + e^{-\lambda \, t/\ve^2} \sum_{n \geq 1} \left ( \frac{\lambda \, t}{ \ve^2 }\right )^n \frac{1}{n!} \, M_\ve^{*_n} * g_0(v).
  \end{equation}

  \subsection{Fourier Transform}

    The standard approach to solve the heat equation on the whole space is to use Fourier transform.
        This is even more obvious for equation \eqref{kin}. Indeed, passing to Fourier variables in this equation shows that  the Fourier transform of $g_\ve$ is given by
    \begin{equation}
      \label{eqGFourier}
      \widehat{g_\ve}(\xi, t) = \widehat{g_0}(\xi) \, \exp\left(- A_\ve(\xi) \, t\right), \qquad \forall (\xi, t) \in \RR \times \RR_+,
    \end{equation}
    where $\widehat g_0$ is the Fourier transform of the initial density and
    \begin{equation}
      \label{defAeps}
      A_\ve(\xi) := \lambda \, \frac{1 -\widehat{M_\ve}(\xi) }{\ve^2}.
    \end{equation}
    Since $M_\ve$ is a probability density, $|\widehat{M_\ve}(\xi)| \le 1$ for all $\xi \in \RR$, which implies that the real part of $A_\ve$ is nonnegative.
    Moreover, thanks to condition \fer{norm}, if $M_\ve$ possesses more than two moments bounded, one can write
    \[
      A_\ve(\xi) := \frac{\lambda\gamma^2}2 \xi^2 + o( \xi^2),
    \]
    {where the $o( \xi^2)$ term is considered for $\xi \to 0$ (which will be the case until the end of the article).}
    Finally, it is possible to invert the Fourier transform \eqref{eqGFourier} and obtain {in a distributional sense}
    \be \label{eqGeps}
      g_\ve(v, t) = g_0 * P_\ve(\cdot, t) (v),
    \ee
    where we set
    \begin{equation*}
      P_\ve(v,t) = \exp\left\{- A_\ve(\xi) t\right\}^\vee(v).
    \end{equation*}

  Clearly,  representations \eqref{eqGpowSerie} and \eqref{eqGeps} have to coincide. Indeed, applying a Fourier transform to \eqref{eqGpowSerie} yields
  \begin{align}
    \widehat{g_\ve}(\xi,t) & \, = \,  e^{-\lambda \, t/\ve^2} \widehat{g_0}(\xi) + e^{-\lambda \, t/\ve^2} \sum_{n \geq 1} \left ( \frac{\lambda \, t}{ \ve^2 }\right )^n \frac{1}{n!} \, \widehat{M_\ve}^{n}(\xi) \, \widehat{g_0}(\xi) \notag \\
     & \, = \, e^{-\lambda \, t/\ve^2} \left [ \left ( e^{\widehat{M_\ve}(\xi)\lambda \, t / \ve^2} - 1\right ) + 1 \right ]\, \widehat{g_0}(\xi) \label{eqG2pieces} \\
     & \, = \, \exp\left(- A_\ve(\xi) \, t\right)\, \widehat{g_0}(\xi), \notag
  \end{align}
   where $A_\ve$ is given by \eqref{defAeps}.

  \begin{remark}
    \label{remLossSmooth}
    According to relation \eqref{eqG2pieces},  any Rosenau-type approximation to the heat equation is such that the smoothness of the solution is lost, independently of the regularity of the kernel $M_\ve$. Indeed, the fundamental solution to equation \eqref{kin}  (namely the solution to the Cauchy problem obtained from the initial value $g_0 = \delta_0 $) can be represented  in Fourier variable by
    \begin{align}
       \widehat{G_\ve}(\xi,t) & \, = \, e^{-\lambda \, t/\ve^2} \left ( e^{\widehat{M_\ve}(\xi)\lambda \, t / \ve^2} - 1\right ) + e^{-\lambda \, t/\ve^2} \notag \\
         & \, =: \, G_1(\xi,t) + G_2(t). \label{eqG1G2}
    \end{align}
    If $M_\ve \in L^1$, then according to Riemann-Lebesgue Lemma, $\left |\widehat{M_\ve}(\xi)\right | \to 0$ when $|\xi| \to \infty$ and then $\left |G_1(\xi,t) \right | \to 0$. The loss of smoothness comes from  $G_2(t)$ which is the Fourier transform of a Dirac mass (of weight depending on time). According to \fer{eqG1G2}, the singularity disappears exponentially  when $\ve \to 0$ or time goes to infinity.
      \end{remark}

    If one considers the Rosenau kernel \eqref{Max} with $\lambda = \sigma^2$,
    \begin{equation*}
      \widehat{M_\ve}(\xi) = \frac{1}{1+(\ve \sigma)^2 \xi^2},
    \end{equation*}
    and for all $t, \ve > 0$,
    \[G_1(\xi,t) = e^{-\sigma^2 \, t/\ve^2} \left ( \exp\left (\frac{\sigma^2 \, t}{\ve^2} \frac{1}{1+(\ve \sigma)^2 \xi^2}\right ) - 1 \right ). \]
    Thus $G_1(, \xi, t)$ decreases exponentially fast towards zero when $|\xi| \to \infty$, and the loss of smoothness comes only from $G_2$.

  \begin{figure}
    \newrgbcolor{darkblue}{0 0 0.6}
      \psset{xunit=1.5cm,yunit=1.4cm,algebraic=true,dotstyle=o,dotsize=3pt 0,linewidth=0.8pt,arrowsize=3pt 2,arrowinset=0.25}
        \begin{pspicture*}(-1.,-1.95)(4.,2)
            \psaxes[xAxis=true,yAxis=true,labels=x,Dx=1,Dy=.5,ticksize=-2pt 0,subticks=2]{->}(0,0)(-0.75,-1.75)(3.65,1.95)
            \rput[br](-.15,0.45){$\ve \sigma$}
            \rput[br](-.15,-0.58){$- \ve \sigma$}
            \rput[br](-.15,0.95){$2 \, \ve \sigma$}
            \rput[br](-.15,-1.08){$-2 \, \ve \sigma$}
            \rput[br](-.15,1.45){$3 \, \ve \sigma$}
            \rput[br](-.15,-1.58){$-3\, \ve \sigma$}
            \psline(0,0)(1,0.5)
            \psline(0,0)(1,-0.5)
            \psline(1,0.5)(2,1)
            \psline(1,0.5)(2,0)
            \psline(1,-0.5)(2,0)
            \psline(1,-0.5)(2,-1)
            \psline(2,1)(3,1.5)
            \psline(2,1)(3,0.5)
            \psline(2,0)(3,0.5)
            \psline(2,0)(3,-0.5)
            \psline(2,-1)(3,-0.5)
            \psline(2,-1)(3,-1.5)
            \rput[tl](3.75,0.19){$n$}
            \psdots[dotsize=4pt 0,dotstyle=*,linecolor=darkblue](0,0)
            \psdots[dotsize=4pt 0,dotstyle=*,linecolor=darkblue](1,0.5)
            \psdots[dotsize=4pt 0,dotstyle=*,linecolor=darkblue](1,-0.5)
            \psdots[dotsize=4pt 0,dotstyle=*,linecolor=darkblue](2,1)
            \psdots[dotsize=4pt 0,dotstyle=*,linecolor=darkblue](2,0)
            \psdots[dotsize=4pt 0,dotstyle=*,linecolor=darkblue](2,-1)
            \psdots[dotsize=4pt 0,dotstyle=*,linecolor=darkblue](3,1.5)
            \psdots[dotsize=4pt 0,dotstyle=*,linecolor=darkblue](3,0.5)
            \psdots[dotsize=4pt 0,dotstyle=*,linecolor=darkblue](3,-0.5)
            \psdots[dotsize=4pt 0,dotstyle=*,linecolor=darkblue](3,-1.5)
        \end{pspicture*}
        \caption[Convolution step of the central difference kernel in a Rosenau approximation]{Schematic representation of the convolution steps of the central difference kernel. Each node is a Dirac mass, weighted by binomial coefficients.}
        \label{figGraph}
    \end{figure}
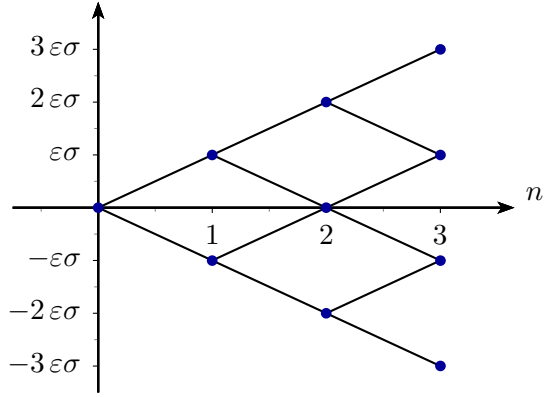

   On the contrary, if one considers the central differences kernel \eqref{CDKernel} with $\lambda = 2$ and $\sigma > 0$, then $ \widehat{M_\ve}(\xi) = \cos(\ve \sigma \xi)$ and for all $t, \ve > 0$,
    \[G_1(\xi,t) = e^{- 2 t/\ve^2} \left ( \exp\left (\frac{2 t}{\ve^2} \cos(\ve \sigma \xi)\right ) - 1 \right ), \]
    which does not converge at infinity. Then, if the kernel has some singularities, the solution to the Rosenau equation has at most the regularity of the initial condition. This enlightens a substantial difference of the Rosenau approximation with respect to the heat equation, in which the diffusion operator is such that the solution is instantaneously smoothed out.
    The representation \eqref{eqGpowSerie} can be used to understand the behavior in time of the solution obtained from the central differences kernel. Indeed, each convolution by $M_\ve$ consists in creating a weighted sum of each point (\ie Dirac masses) of the grid (see Figure \ref{figGraph}):
    \begin{align*}
      & M_\ve * \delta_0 (v) \, = \, \frac{1}{2} \left ( \delta_0 (v- \ve \sigma)  + \delta_0(v + \ve \sigma) \right ), \\
      & M_\ve^{*_2} * \delta_0(v ) \, = \, \frac{1}{4} \, \delta_0(v-2 \, \ve \sigma) + \frac{1}{2} \, \delta_0(v) + \frac{1}{4} \, \delta_0(v-2 \, \ve \sigma) , \\
      & M_\ve^{*_3} * \delta_0(v ) \, = \, \frac{1}{8} \, \delta_0(v-3 \, \ve \sigma) + \frac{3}{8} \, \delta_0(v - \ve \sigma)  + \frac{3}{8} \, \delta_0(v + \ve \sigma)+ \frac{1}{8} \, \delta_0(v + 3\, \ve \sigma)
    \end{align*}
    and so on. Thus, the fundamental solution of the central difference scheme is given by
    \begin{equation*}
      g_\ve(v,t) \, = \, e^{-2 t/\ve^2} \delta_0(v) + e^{-2 t/\ve^2} \sum_{n \geq 1} \left ( \frac{2 t}{ \ve^2 }\right )^n \frac{1}{n!} \, g^{(n)}_\ve(v),
    \end{equation*}
    where, for $n \in \NN$ we defined
    \begin{align*}
      g^{(2 n)}_\ve(v)  & \, = \, \left (\frac{1}{2}\right )^n \sum_{k=0}^n \begin{pmatrix} 2 n \\ 2 k \end{pmatrix} \left [ \delta_0(v+ 2 \, k \ve \sigma) + \delta_0(v - 2 \, k \ve \sigma) \right ], \\
      g^{(2 n + 1)}_\ve(v)  & \, = \, \left (\frac{1}{2}\right )^{n+1} \sum_{k=0}^n \begin{pmatrix} 2 n + 1 \\ 2 k +1 \end{pmatrix} \left [ \delta_0(v+ (2 k +1)  \ve \sigma) + \delta_0(v - (2 k + 1) \ve \sigma) \right ].
    \end{align*}

\section{Asymptotic Behavior of Solutions to the Rosenau Equation}
\label{secAsympt}

In Section \ref{secKinetic} the similarities between the Rosenau-type kinetic equation and the linear diffusion equation have been enlightened.
In particular,  as well known from the classical literature on numerical approximation of partial differential equations, equation \fer{kin} provides a consistent approximation of the heat equation in a fixed time interval.
It is not known, however, if equation \fer{kin} still realizes a good approximation for large times, and,  in case of a positive answer, in which way the difference between the solutions of the Rosenau and diffusion equations can be estimated with  respect both to time and the parameter $\ve$.

In this Section, we will furnish a partial answer to this question. To clarify our intent,  we will briefly resume various well-known facts about the large-time behavior of the solution to the heat equation posed in the whole space.
Let $g=g(v,t)$ be a solution to the heat equation \eqref{heat} and define the \emph{heat kernel} $\Omega_\sigma$ by
\begin{equation}
  \label{heatKernel}
  \Omega_\sigma(v,t) := \frac{1}{\sqrt{4 \pi \sigma^2 t}} \exp\left ( - \frac{v^2}{4 \sigma^2 t} \right ), \quad \forall (v,t) \in \RR \times \RR_+.
\end{equation}
This fundamental solution of \eqref{heat} represents an intermediate asymptotics of a large class of solutions to the heat equation.
The recent review article~\cite{Bartier:2011} gives a precise state of the art on this topic.
To make this concept more precise, we define the {Boltzmann \emph{entropy}} of $f$ as
\begin{equation*}
  \mathcal H(f) := \int_\RR f(v) \, \log f(v) \,  dv.
\end{equation*}
Then it can be shown  (see \eg \cite{Toscani:1996}) that $g(v,t)$ behaves as the heat kernel when $t \to \infty$, provided that the initial condition $g_0$ is of finite kinetic energy and entropy:
\begin{equation*}
  g_0 \in L^1_2 \quad \text{and} \quad  \mathcal{H}(g_0) < \infty.
\end{equation*}
Moreover, the rate of convergence towards the fundamental solution can be computed in $L^1$ norm
\begin{equation} \label{rateCvHeat}
\left \| g(t) - \Omega_\sigma(t) \right \|_{L^1} \leq \frac{C}{\sqrt{1+2t}},
\end{equation}
where $C$ is an explicit constant.
The bound \eqref{rateCvHeat} is sharp. A marked improvement of the constant in \fer{rateCvHeat} has been recently obtained in \cite{arnold:2008}, by selecting well parametrized Gaussian functions, characterized either by mass centering or by fixing the second moments or the covariance matrix of the solution.

Most of these results follow by considering the equivalence between  the heat equation and the linear Fokker-Planck equation
\begin{equation} \label{eqFP}
\frac{\partial u}{\partial t} = \frac{\partial}{\partial v} \left ( v u + \sigma^2\frac{\partial u}{\partial v} \right ).
\end{equation}
which has a stationary solution given by the Gaussian
\be \label{defGaussian}
  \omega_\sigma(v) := \Omega_\sigma (v, t = 1).
\ee
Setting $T(t) := e^{2 t - 1} / 2$, we can in fact consider the change of variables
\begin{equation} \label{selfsimScal}
u(v,t) := e^t g\left (e^t v, T(t)\right ).
\end{equation}
Then $u(v,t)$ is solution to the Fokker-Planck equation \eqref{eqFP} as soon as $g$ is a solution to the heat equation \eqref{heat}. The converse also holds true.
It has been shown  in \cite{Carrillo:1998} that the solution to the linear Fokker-Planck equation converges for large times to the stationary solution $\omega_\sigma$ and the  following asymptotic behavior holds
\begin{equation*}
\| u(t) - \omega_\sigma  \|_{L^1} \leq C e^{-t}, \, \forall t \geq 0.
\end{equation*}
Using the self-similar scaling \eqref{selfsimScal}, it is clear that this result is equivalent to the rate \eqref{rateCvHeat}, and once more, the bound is sharp.
We shall use a similar rescaling in our study of the intermediate asymptotics of equation \eqref{kin}.

Our goal in the following is to prove that the same kind of result holds for the Rosenau-type approximation to the heat equation.
Due to the generality of this approximation, which depends of the background distribution $M$, we will obtain a weaker convergence (with respect to the $L^1$ distance) to the fundamental solution, as discussed in Section \ref{secRepr}, Remark \ref{remLossSmooth}.
The $L^1$ distance will be here substituted by the suitable weaker Fourier based metrics $d_s$ described in Section \ref{subFuncFram}, which are particularly adapted to the convolution structure of the kinetic equation \eqref{eqBoltzResc}.
This family of metrics has been introduced in the article \cite {Gabetta:1995} to study the trend to equilibrium of solutions to the space homogeneous Boltzmann equation for Maxwell molecules, and subsequently applied to a variety of problems related to kinetic models of Maxwell type.
For a more detailed description, we address the interested reader to the recent lecture notes \cite{Carrillo:2007}.

These metrics can be easily and fruitfully applied to the study of the large-time behavior of the heat equation. Indeed, we have seen that in Fourier variables, the solution to the heat equation \eqref{heat} with initial condition $g_0$ is
\begin{equation*}
  \widehat g(\xi,t) = \widehat{g_0}(\xi) \, \exp \left (-\sigma^2 \xi^2 \,t \right ).
\end{equation*}
Thus, if $g_0$ is of finite mass, the solution will converge pointwise towards $0$ when $t \to \infty$.  A way to consider a nontrivial limit distribution is to make the change of variable $\xi \to V(t) \, \xi$ with $V(t) := (1+t)^{-1/2}$. Then, the scaled distribution will converge towards the Gaussian distribution
\[
  \widehat{g_0}(0) \, \exp \left (-\sigma^2 \xi^2 \right ),
\]
and one can improve the rate \eqref{rateCvHeat} in Fourier distance, as soon as the initial datum possesses finite moments of certain order equal to those of the Gaussian   \cite{Goudon:2002}. From now on, let us set
\be
  \label{eqGaussFourier}
  \widehat\omega_\sigma(\xi) \, = \, \exp \left (-\sigma^2 \xi^2 \right ),
\ee
the Fourier transform of the Gaussian distribution \eqref{defGaussian}. We have the following result.

\begin{proposition}
    \label{propAsymptEx}
    Let $V(t) := (1+t)^{-1/2}$. For a given $s > 1$, let $0 \leq g_0 \in L^1_s$ a distribution of unit mass.
    Then, if $g(v,t)$ is the unique solution to the Cauchy problem \eqref{heat} with initial condition $g_0$,  the scaled distribution $h(v, t) := V(t)^{-1}\, g \left( V(t)^{-1}\, v, t\right)$ is such that, for all $t \geq 0$
    \begin{equation} \label{ineqdsExact}
      d_s \left ( h(t), \omega_\sigma \right ) \leq \frac{1}{(1+t)^{s/2}}d_s(g_0,\omega_\sigma).
    \end{equation}
\end{proposition}

\begin{proof}
    Thanks to the scaling properties of the Fourier transform one has
    \begin{align*}
        d_s \left ( h(t), \omega_\sigma \right ) & = \sup_{\xi \in \RR \setminus 0} \, \frac{ \left | \, \widehat g(V(t) \xi, t)-\widehat{\omega_\sigma}(\xi) \right | }{|\xi|^s}  \\
        & = \sup_{\xi \in \RR \setminus 0} \, \frac{1}{|\xi|^s}\left |\widehat{g_0}(V(t) \xi) - \exp \left (-\sigma^2 \xi^2 V^2(t) \right )\right |  \left|\exp \left (- \sigma^2 \xi^2 \frac{t }{1+t} \right )\right| \\
        & \leq  V(t)^s \sup_{\xi \in \RR \setminus 0} \, \frac{1}{|\xi|^s}\left | \widehat{g_0}(\xi) - \exp \left (-\sigma^2 \xi^2 \right )\right | .
    \end{align*}
\end{proof}

\begin{remark}
  Taking a Dirac mass as the initial condition in the Cauchy problem  for equation \ref{heat} shows that these rates are optimal.
\end{remark}

\subsection{Approximate Solutions with Finite Energy}
  \label{subFinKinEn}

To recover the asymptotic behavior of the solution to the Rosenau-type approximation \fer{Ros}, we are going to apply a technique similar to that used in Proposition \ref{propAsymptEx}.
Let us consider again the scaling $\xi \to V(t) \xi$ with $V(t) := (1+t)^{-1/2}$, and define
\begin{equation*}
h(v, t) := V(t)^{-1}\, g \left( V(t)^{-1}\, v, t\right)
 \text{ and } h_\ve(v, t) := V(t)^{-1}\, g_\ve \left( V(t)^{-1}\, v, t\right).
\end{equation*}
We choose in this section an initial datum $0 \leq g_0 \in L^1_2$ such that
\be\label{norm2}
  \int_\RR  g_0(v) \, dv = 1, \quad\int_\RR v \, g_0(v) \, dv = 0, \quad  \int_\RR v^2 g_0(v) \, dv  = E < \infty,
 \ee
and $\lambda \, \gamma^2/2 = \sigma^2$. Then the quantity  $d_2 \, (g_0, \omega_\sigma)$ is bounded, and, according to \eqref{ineqdsExact} one obtains
\begin{equation} \label{ineqd2hepsomega}
d_2 \, (h_\ve(t), \omega_\sigma) \leq \frac{1}{1 + t}d_2 \, (g_0, \omega_\sigma) + d_2 \, (h_\ve(t), h(t)), \, \forall t \geq 0.
\end{equation}
Hence, we have to estimate the Fourier distance between the exact scaled solution $h$ and its approximate counterpart $h_\ve$. We shall use for this the representation~\eqref{eqGFourier} of $g_\ve$ in Fourier variables. Let us introduce a positive parameter $R$. Since the mass of the initial datum $g_0$ is equal to $1$, one has
\begin{align*}
d_2 \, (h_\ve(t), h(t)) & \leq \sup_{\xi \in \RR \setminus 0} \,  \frac{1}{\xi^2} \left | \exp \left \{ - A_\ve \left ( \frac{\xi}{\sqrt{1+t}} \right ) t \right  \} - \exp\left \{ - \sigma^2 \xi^2 \frac{t}{1 + t} \right \} \right | \\
& \leq \frac{2}{R^2} + \sup_{|\xi| \leq R} \, \frac{1}{\xi^2} \left | \exp \left \{ - A_\ve \left ( \frac{\xi}{\sqrt{1+t}} \right ) t \right  \} - \exp\left \{ - \sigma^2  \xi^2 \frac{t}{1 + t} \right \} \right | \\
& \leq  \frac{2}{R^2} + \sup_{|\xi| \leq R} \, \frac{1}{\xi^2}  \left | \exp \left \{ - \sigma^2  \xi^2 \frac{t}{1 + t} \left [ \frac{1 + t}{ \sigma^2 \xi^2}A_\ve \left ( \frac{\xi}{\sqrt{1+t}} \right )-1 \right ] \right \} - 1 \right |.
\end{align*}
Using the elementary inequality $|1- e^{-x}| \leq |x|$, one finally has
\begin{equation}\label{ineqd2hepsh}
d_2(h_\ve(t), h(t)) \leq  \frac{2}{R^2} + \sup_{|\xi| \leq R} D_\ve(\xi, t),
\end{equation}
where
\begin{equation*}
D_\ve(\xi, t) := \frac{t}{(\ve \xi )^2} \left | \lambda\, \left( 1 - \widehat{M_\ve} \left ( \frac{\xi}{\sqrt{1+t}} \right )\right) - \frac{(\ve \sigma \xi )^2}{1 + t} \right |.
\end{equation*}

In order to compute this last quantity and to obtain the rate of convergence, let us specify both the background distribution and the constants $\gamma$ and $\lambda$.

\paragraph{\textbf{Central Differences.}}

    As in Remark \ref{cds}, let us set $\lambda = 2$ and
    \begin{equation*}
      M_\ve(v) = \frac{1}{2} [ \delta_0(v+ \ve \sigma) + \delta_0(v-\ve \sigma) ].
    \end{equation*}
    Then $ \widehat{M_\ve}(\xi) = \cos(\ve \sigma \xi)$, and thanks to Taylor's theorem, there exists $\theta \in (0,1)$ such that
    \begin{equation}
        D_\ve(\xi, t) = \frac{t}{3} \frac{\ve \sigma \xi}{(1+t )^{3/2}} \left | \sin \left (\theta\frac{\ve \sigma \xi}{\sqrt{1+t}} \right )\right |
        \leq \frac{t}{(1+t)^2} \frac{(\ve \sigma \xi)^2}{3}.
        \label{ineqDepsCD}
    \end{equation}
    Gathering inequalities \eqref{ineqd2hepsh} and \eqref{ineqDepsCD} yields for $R > 0$ the inequality
    \begin{equation*}
      d_2(h_\ve(t), h(t)) \leq \frac{2}{R^2} + \frac{t}{(1+t)^2} \frac{(\ve \sigma R)^2}{3}.
    \end{equation*}
    It just remains to optimize on $R$ to obtain
    \begin{equation}
      \label{ineqd2hepshCD}
      d_2(h_\ve(t), h(t)) \leq \sqrt{\frac{3 \sigma^2}{2}} \ve \frac{\sqrt{t}}{1+t}.
    \end{equation}
    Inequality \eqref{ineqd2hepshCD} shows that the second term in \eqref{ineqd2hepsomega} has a slower rate of decay with respect to the first one.
    In contrast to the exact solution $h$,  which is known to converge in $d_2$ towards the self-similar one at rate $(1+t)^{-1}$ \cite{Goudon:2002}, we found that the approximate solution converges towards the exact one at rate $(1+t)^{-1/2}$. We just stated the following
	
    \begin{proposition}
        {Let $0 \leq g_0 \in L_2^1$ satisfy \eqref{norm2}}.
        If $g_\ve$ is solution to the associated Cauchy problem for the Rosenau-type approximation \fer{Ros} with the central difference kernel \fer{CDKernel}, then the scaled distribution
         \[h_\ve(t, v) :=  V(t)^{-1} \, g_\ve \left ( V(t)^{-1} v, t\right )\]
        for  $V(t) := (1+t)^{-1/2}$ verifies for all $t \geq 0$
        \begin{equation*}
        d_2 \left ( h_\ve(t), \omega_\sigma \right ) \leq \frac{1}{1 + t}d_2(g_0, \omega_\sigma) +  \sqrt{\frac{3 \sigma^2}{2}} \ve \frac{\sqrt{t}}{1+t}.
        \end{equation*}
    \end{proposition}

\paragraph{\textbf{Rosenau Regularization.}}

    In this case the Fourier transformed ``Maxwellian'' reads
    \begin{equation*}
      \widehat{M_\ve}(\xi) = \frac{1}{1+(\ve \sigma)^2 \xi^2},
    \end{equation*}
    and $\lambda = \sigma^2$. We obtain
    \begin{equation*}
      D_\ve(\xi, t) = \frac{(\ve \sigma)^2t}{1+t} \left | \frac{\xi^2}{1 + t + (\ve \sigma)^2 \xi^2} \right|.
    \end{equation*}
    Inserting this expression in inequality \eqref{ineqd2hepsh} gives for $R > 0$
    \begin{equation*}
      d_2(h_\ve(t), h(t)) \leq \frac{2}{R^2} + \frac{t}{(1+t)^2} \frac{(\ve \sigma R)^2}{3}.
    \end{equation*}
    It just remains to optimize on R to obtain the rate of convergence in $d_2$ norm of $h_\ve$ towards $h$:
    \begin{equation}
      \label{ineqd2hepshRos}
      d_2(h_\ve(t), h(t)) \leq \sqrt{\frac{\sigma^2}{2}} \ve \frac{\sqrt{t}}{1+t}.
    \end{equation}
    Hence, this approximation has exactly the same order of convergence than the one given by the central differences kernel. We  proved

    \begin{proposition}
        {Let $0 \leq g_0 \in L_2^1$ satisfy \eqref{norm2}}.
        If $g_\ve$ is solution to the associated Cauchy problem for the Rosenau-type approximation \fer{Ros} with the Rosenau kernel \fer{Max}, then the scaled distribution
        \[ h_\ve(t, v) :=  V(t)^{-1} \, g_\ve \left ( V(t)^{-1} v, t\right )\]
        for  $V(t) := (1+t)^{-1/2}$ verifies, for all $t \ge 0$
        \begin{equation*}
          d_2 \left ( h_\ve(t), \omega_\sigma \right ) \leq \frac{1}{1 + t}d_2(g_0, \omega_\sigma) +  \sqrt{\frac{ \sigma^2}{2}} \ve \frac{\sqrt{t}}{1+t}.
        \end{equation*}
    \end{proposition}

\begin{remark}
  Note that as a by-product of our analysis of the large-time behavior of the solution to the Rosenau-type approximation of the diffusion equation we proved that the $d_2$-distance between the solution of the diffusion equation and its numerical approximation by the central difference scheme is uniformly bounded in time with respect to the parameter $\ve$, as given by \fer{ineqd2hepshCD}.
\end{remark}

\subsection{Approximate Solution with Finite Fourth Order Moment}
\label{subFourth}

In Subsection \ref{subFinKinEn}, the convergence results required the boundedness of the $d_2$-distance between the initial value of the Rosenau equation and the Gaussian density. This boundedness has been achieved by assuming that the initial datum is of finite energy.
Indeed, moments play an important role with respect to convergence. We will improve here the results of Subsection \ref{subFinKinEn},  provided that the initial condition has finite moments up to order four. We will see that this allows to compute the rate of convergence (in time) towards self-similarity with very few assumptions on the kernel $M_\ve$.

Without loss of generality (thanks to scaling and translational invariance), we can choose the initial distribution {$0 \leq g_0 \in L_4^1$} to satisfy
\begin{equation}
  \label{eqMomOrd4}
  \int_\RR g_0(v)\, \varphi(v) \, dv = (1, 0, 1, 0, \mu),
\end{equation}
where $\varphi(v) = (1, v, v^2/2, v^3, v^4)$, and $\mu$ is a positive constant. If $g$ is solution to the Cauchy problem for equation \eqref{heat} with initial condition $g_0$, then on the one hand
\begin{equation} \label{eqMomG}
\int_\RR g(v, t)\, \varphi(v) \, dv = (1, 0, 1 + t \sigma^2, 0, \mu + 24 t + 12 \sigma^2 t^2 ).
\end{equation}
On the other hand, a solution $g_\ve$ to equation \eqref{Ros} with same initial datum verifies
\begin{equation}\label{eqMomGeps}
\int_\RR g_\ve(v, t)\, \varphi(v) \, dv = (1, 0, 1 + t \sigma^2, 0, \mu + (24 + B_\ve) t + 12 \sigma^2 t^2) ),
\end{equation}
provided that $M_\ve$ follows assumption \eqref{norm} and where
\begin{equation*}
B_\ve := \frac{2}{\ve^2}\int_\RR M_\ve(v) v^4 \, dv.
\end{equation*}
We can see that the fourth order moment of the exact and approximate solutions differ at time $t > 0$ of a quantity equal to $B_\ve t$.
Keeping the same notations of Subsection \ref{subFinKinEn}, we will use this fact to compute the $d_3$ distance between the scaled approximate solution $h_\ve$ and the rescaled exact one $h$. Indeed, it is well know  that if $k \in \NN$, the $k^{th}$ moment of a distribution correspond to the
$k^{th}$ derivative of its Fourier transform. Then $\widehat{g}(\cdot, t)$ and $\widehat{g_\ve}(\cdot, t)$ are at least four times differentiable and according to Taylor's theorem \eqref{eqMomG}--\eqref{eqMomGeps} imply
\begin{align*}
  & \widehat{g}(\xi, t) = 1 + \frac{1+t \sigma^2}{2}\xi^2 + \frac{1 + 24 t + 12 \sigma^2 t^2}{24} \xi^4 + \mathcal{O}\left (|\xi|^5\right ), \\
  & \widehat{g_\ve}(\xi, t) = 1 + \frac{1+t \sigma^2}{2}\xi^2 + \frac{1 + (24 + B_\ve) t + 12 \sigma^2 t^2}{24} \xi^4 + \mathcal{O}\left (|\xi|^5\right ).
\end{align*}
Thus, one has for $R > 0$
\begin{align}
  d_3 \, (h_\ve(t), h(t)) & = \sup_{\xi \in \RR \setminus 0} \frac{1}{|\xi|^3} \left | \widehat{g_\ve} \left (\frac{\xi}{\sqrt{1+t}}, t \right ) - \widehat{g} \left (\frac{\xi}{\sqrt{1+t}}, t \right )\right | \notag \\
  & \leq \frac{2}{R^3} + \frac{B_\ve t}{24(1+t)^2} R. \label{ineqd3hepsh}
\end{align}
Optimizing \eqref{ineqd3hepsh} over $R$ and gathering the result with inequality \eqref{ineqdsExact} (with $s=3$) gives

\begin{theorem}
  \label{thmConvD3}
    {Let $0 \leq g_0 \in L^1_4$  satisfy \eqref{eqMomOrd4}. }
    If $g_\ve$ is solution to the associated Cauchy problem for Rosenau approximation \fer{Ros} with a kernel $M_\ve$ verifying \eqref{norm}, then the scaled distribution
    \[h_\ve(t, v) :=  V(t)^{-1} \, g_\ve \left ( V(t)^{-1} v, t\right )\]
    for  $V(t) := (1+t)^{-1/2}$ is such that, for all $t \ge 0$
    \begin{equation*}
    d_3 \left ( h_\ve(t), \omega_\sigma \right ) \leq \frac{1}{(1+t)^{3/2}} d_3(g_0, \omega_\sigma) +  \frac{13 \sqrt{2}}{24} (B_\ve)^{3/4} \left (\frac{\sqrt{t}}{1+t}\right )^{3/2},
    \end{equation*}
    where $\omega_\sigma$ is the Gaussian distribution \eqref{eqGaussFourier} and
    \begin{equation*}
    B_\ve := \frac{2}{\ve^2}\int_\RR M_\ve(v) v^4 \, dv.
    \end{equation*}
\end{theorem}

\begin{remark}
 The result of Theorem \ref{thmConvD3} indicates in a clear way that the rate of convergence towards the fundamental solution improves as soon as the initial value has a
 a higher number of moments equal to those of the Gaussian. In the present case, with four moments the difference in Fourier norm between solutions converges to zero at rate $(1+t)^{-3/4}$. This suggests that the rate of convergence approaches the optimal rate when the number of finite moments of the initial value approaches infinity. This aspect of the problem  is deeply connected to the analogous one studied in the central limit theorem, which is referred to as Berry-Esseen estimates.
 We remark that Berry-Esseen-like estimates for Fourier metric based metrics have been studied in \cite{Goudon:2002}.
\end{remark}

Let us  now specify the kernel $M_\ve$, to check if there are essential differences among the rate of convergence in $\ve$.

The choice of the central differences kernel
\begin{equation*}
M_\ve(v) = \frac{1}{2} [ \delta_0(v+ \ve \sigma) + \delta_0(v-\ve \sigma) ]
\end{equation*}
gives $B_\ve = 2 \ve^2 \sigma^4$. Then
\begin{equation*}
d_3 \left ( h_\ve(t), \omega_\sigma \right ) \leq \frac{1}{(1+t)^{3/2}} d_3(g_0, \omega_\sigma)+ C_1 \, \ve^{3/2} \left (\frac{\sqrt{t}}{1+t}\right )^{3/2},
\end{equation*}
where $C_1$ is a positive constant depending on $\sigma$. Compared to \eqref{ineqd2hepshCD},  the rate in $\ve$ is improved by a power $1/2$.

If we consider now the Rosenau kernel
\begin{equation*}
\widehat{M_\ve}(\xi) = \frac{1}{1+(\ve \sigma)^2 \xi^2},
\end{equation*}
$B_\ve = K \ve^3$ for a nonnegative constant $K$. Then
\begin{equation*}
d_3 \left ( h_\ve(t), \omega_\sigma \right ) \leq \frac{1}{(1+t)^{3/2}} d_3(g_0, \omega_\sigma) + C_2 \, \ve^{9/4} \left (\frac{\sqrt{t}}{1+t}\right )^{3/2},
\end{equation*}
where $C_2$ is a positive constant depending on $\sigma$.

\begin{remark}
{We can see here that the Rosenau kernel gives a better convergence rate in $\ve$ towards the exact solution, compared to the central differences kernel. This improvement also depends of the number of finite moments.}
This shows that the number of finite moments plays an essential role also in connection with the order of convergence in $\ve$. Indeed in formulas \eqref{ineqd2hepshCD} and \eqref{ineqd2hepshRos}, which refer to the situation in which only moments up to order two have been considered, this difference between the approximations due to the two kernels was not evident.
\end{remark}

\subsection{Strong Convergence of a ``Regularized'' Approximate Solution}

Let us assume that the kernel $M_\ve$ belongs to $L^1(\RR)$. The analysis of Section \ref{secAsympt} shows that, in a suitable scaling which allows to maintain the energy of the solution bounded, there is convergence in Fourier distance towards a Gaussian function. This result, however, cannot be directly used to conclude that the solution to the Rosenau kinetic equation converges towards the fundamental solution {to the heat equation}, as time goes to infinity.
In fact, since the $d_s$-metric is not scaling invariant, and the decay in time of the distance found in Theorem \ref{thmConvD3} is of order $t^{-3/4}$, by reverting to the original variables, the decay in time disappears.
Nevertheless, these weak convergence estimates can be fruitfully employed to prove that
the solution to Rosenau equation can be split into two parts, one of which is exponentially decaying to zero with respect to both time $t$ and the parameter $\ve$,
while the other converges strongly in time towards the fundamental solution to the heat equation at a lower rate. This result clarifies the nature of the Rosenau-type approximation to the heat equation. This approximation produces a solution which essentially consists of two parts, which differs for their regularity.
The part with lower regularity (essentially the first term in \eqref{eqGpowSerie} with the same regularity of the initial datum) is rapidly decaying to zero. The regular part is shown to behave like the classical solution to the heat equation, and decays towards the self-similar Gaussian solution for large times.
The result will be achieved by making a heavy use of the representations given in Section \ref{secRepr}. In particular, we will write the fundamental solution \fer{eqG1G2} by splitting it in two parts
\begin{align}
  \widehat{G_\ve}(\xi,t) & \, = \, e^{-\lambda \, t/\ve^2} \left ( e^{\widehat{M_\ve}(\xi)\lambda \, t / \ve^2} - \left (1- \widehat M_\ve(\xi)\right ) \right ) + \left (1- \widehat M_\ve(\xi) \right ) e^{-\lambda \, t/\ve^2} \notag \\
         & \, =: \widehat P_{\ve, reg}(\xi, t) + \left (1- \widehat M_\ve(\xi)\right ) e^{-\lambda \, t/\ve^2}. \label{eqPveReg}
\end{align}
As we shall see, this splitting separates in a natural way the singular part of the kernel from its \emph{regularized} part $P_{\ve, reg}$.

Let $g$ (respectively $g_{\ve, reg}$) a solution to the heat equation \eqref{heat} (resp. a solution to the Rosenau equation \eqref{kin} obtained by convolution with the regularized kernel). In other words
\begin{gather}
  g(v,t) \, = \, g_0(v) * \Omega_\sigma(\cdot, t), \notag \\
  g_{\ve, reg}(v,t) \, = \,  g_0(v) * P_{\ve, reg}(\cdot, t). \label{eqGveReg}
\end{gather}
We remark that, resorting to formula \eqref{eqGpowSerie}, one recovers that $ g_{\ve, reg}(v,t)$ coincides with the Wild sum expansion in which the zero-order term is substituted by the regular term $M_\ve* g_0(v)$.
Hence, $ g_{\ve, reg}(v,t)$ is nothing but a regularized version of the Wild sum expansion \eqref{eqGpowSerie}. Since the initial datum $g_0(v)$ is of unit mass,
\begin{align*}
  \left \|g(t) -  g_{\ve, reg}(t)\right \|_{L^1} = & \ \int_\RR \left | \int_\RR \left ( \Omega_\sigma(v-w, t) - P_{\ve, reg}(v-w, t) \right )\, g_0(w) \, dw \right | dv\\
   \leq & \ \left \|\Omega_\sigma(t) - P_{\ve, reg}(t)\right \|_{L^1}.
\end{align*}
On the other hand, if $\lambda = \sigma^2$
\begin{align*}
  \left| \widehat P_{\ve, reg}(\xi,t) - \widehat \Omega_\sigma(\xi, t)\right| = & \ \left| e^{-\sigma^2 \, t/\ve^2} \left ( e^{\widehat{M_\ve}(\xi)\, \sigma^2 \, t / \ve^2}- \left (1- \widehat M_\ve(\xi)\right ) \right ) - e^{-\sigma^2|\xi|^2 t}\right| \\
   \leq & \ \left| e^{-\sigma^2 \, t/\ve^2}  e^{\widehat{M_\ve}(\xi)\, \sigma^2 \, t / \ve^2}- e^{-\sigma^2|\xi|^2 t}\right| + \left|1- \widehat M_\ve(\xi) \right|e^{-\sigma^2 \, t/\ve^2},
\end{align*}
and this implies
 \be\label{ok}
d_2\left (P_{\ve, reg}(t), \Omega_\sigma(t)\right ) \le d_2\left (G_\ve(t), \Omega_\sigma(t)\right ) + d_2\left (M_\ve, \delta_0\right )e^{-\sigma^2 \, t/\ve^2}.
 \ee
By means of the scaling property of the Fourier metric $d_s$,  formula \fer{ok} ensures that, after the scaling $\xi \to \xi/\sqrt{1+t}$, the convergence results (in scaled variables) of Section \ref{secAsympt} guarantee the convergence of the regularized kernel towards the Gaussian fundamental solution.

Now, thanks to the regularity of $P_{\ve, reg}(t)$,  we can consider also convergence of the (regularized) approximate fundamental solution $ P_{\ve, reg}$ towards the heat kernel $\Omega_\sigma $ in stronger norms, typically $L^1$, which are invariant with respect to dilatation. Actually, it is enough to deal with $L^2$ norms. Indeed, it is well known \cite{Carlen:1999, Carrillo:2007} that if $f \in L^2(\RR)$ is of finite kinetic energy, there exists an explicit, nonnegative constant $C_1$ such that
\begin{equation}
  \label{ineqL2toL1}
  \|f\|_{L^1} \leq C_1 \, \|f\|_{L^2}^{4/5}\left ( \int_\RR |v|^2 \, |f(v)| \, dv \right )^{1/5}.
\end{equation}
Note that the second moment of the quantities involved is uniformly bounded only in the scaled variables used in Section \ref{secAsympt}. In original variables, in fact, the second moment of the solution, both of the heat equation and its Rosenau approximation, is increasing linearly with respect to time.
Then, it is possible to interpolate the $L^2$ norm by some appropriate Fourier and homogeneous Sobolev norms.
More precisely, it was shown in \cite{Carlen:1999} that for all $r \in (0,1)$ there exists an explicit, nonnegative constant $C_2=C_2(r)$ such that
\begin{equation}
  \label{ineqDsToL2}
  \|f-g\|_{L^2} \leq C_2 \, d_2\,(f,g)^{2(1-r)} \, \|f-g\|_{\dot H^N}^{2 r},
\end{equation}
where $N = 3 \,(1-r)/r$. Then, according to inequalities \eqref{ineqL2toL1}--\eqref{ineqDsToL2}--\eqref{ok} and to the results of Subsection \ref{subFourth}, we will obtain the strong convergence of $ P_{\ve, reg}(t, \cdot)$ towards $\Omega_\sigma(t, \cdot) $ when $t \to \infty$ if this function belongs to the homogeneous Sobolev space $\dot H^N(\RR)$ for arbitrarily small $N$, and at the same time its Sobolev norm increases in time at a rate such that the right-hand side of inequality \fer{ineqDsToL2} decays to zero as time goes to infinity.

We will deal with the Rosenau kernel with $\lambda = \sigma^2$
\begin{equation*}
  \widehat{M_\ve}(\xi) = \frac{1}{1+(\ve \sigma)^2 \xi^2}
\end{equation*}
which belongs to $\dot H^s(\RR)$ for all $s < 1$. Hence, in scaled variables, for all $s <1$
\begin{align*}
  e^{-\sigma^2 \, t/\ve^2} \left  \| \widehat{M_\ve} \left (\frac\xi{\sqrt{1+t}}\right ) \right \|_{\dot H^s}^{2} & \, = \, e^{-\sigma^2 \, t/\ve^2}\left( \frac{\sqrt{1+t}}{\ve \sigma}\right)^{1+s} \left \| \widehat{M_\ve} (\xi)\right \|_{\dot H^s}^{2} \\
  & \, \le \, C(\ve) \| \widehat{M_\ve}(\xi)\|_{\dot H^s}^{2}.
\end{align*}

According to \eqref{eqPveReg}, it remains to check the smoothness of
\begin{equation*}
  \widehat{\mathcal G_\ve}(\xi, t) :=  e^{\widehat{M_\ve}(\xi)\,\sigma^2 \, t / \ve^2} - 1, \ \forall t \geq 0.
\end{equation*}
Let us set $s < 1$. We have thanks to \emph{de l'Hospital} rule
\begin{align*}
  \lim_{|\xi| \to \infty} \xi^{s+1} \widehat{\mathcal G_\ve}(\xi, t) 
   & \, = \, \lim_{|\xi| \to \infty}  - \frac{\sigma^2 \, t}{(s+1) \, \ve^2} \, \xi^{s+2} \, \widehat{M_\ve}'(\xi) \, e^{\widehat{M_\ve}(\xi)\,\sigma^2 \, t / \ve^2} \\
   & \, = \, \lim_{|\xi| \to \infty} \frac{2 \, \sigma^4 \, t}{s+1} \, \frac{\xi^{s+3}}{(1+(\ve \sigma)^2 \xi^2)^2} \exp \left (  \frac{\sigma^2 \, t}{\ve^2 \, (1+(\ve \sigma)^2 \xi^2)}\right ) = 0.
\end{align*}
Then we have proved that $\xi^s \, \widehat{\mathcal G_\ve}(\xi, t) = o(\xi^{-1}) $ which means according to Riemann criterion that $ \mathcal G_\ve(\cdot, t) \in \dot H^s(\RR)$ for any $t \geq 0$. Detailed computations, reported in Appendix, allow to conclude that, for $s<1$,
 \be\label{diff}
e^{-\sigma^2 \, t/\ve^2} \left \| \widehat{\mathcal G_\ve} \left (\frac\xi{\sqrt{1+t}}\right ) \right \|_{{\dot H^s}} \le C(\ve, s)\,(1+ t)^\delta
 \ee
where the explicitly computable constant $C(\ve, s)$ depends on $\ve$, $s$ and $\widehat M_\ve$, and $\delta >0$ can be taken as small as we want.
Thanks to inequalities \eqref{ok}--\eqref{ineqL2toL1}--\eqref{ineqDsToL2}--\fer{diff} and Theorem \ref{thmConvD3} we obtain the following

\begin{theorem}\label{main1}
  Under the assumptions of Theorem \ref{thmConvD3}, if $g$ is a solution to the heat equation \eqref{heat} and $g_{\ve, reg}$ is given by \eqref{eqGveReg} for the Rosenau kernel \eqref{Max}, then one has
  \begin{equation*}
    \lim_{t \to \infty} \left \|g(t) -  g_{\ve, reg}(t)\right \|_{L^1} = 0.
  \end{equation*}
\end{theorem}

\begin{remark}
Theorem \ref{main1} has various consequences. Coupling the decay of the $d_2$ metric obtained in \fer{ineqd2hepshCD} with inequalities \fer{ineqL2toL1} and \fer{ineqDsToL2}, the regularity of $P_{\ve, reg}$ implies a rate of decay
of the $L^1$-norm smaller or equal to $(1+t)^{1/15}$. On the other hand, the singular part of the fundamental solution $\widehat{G_\ve}(\xi,t)$ converges to zero exponentially both in time and with respect to $\ve$. Hence, after a transitory time, there is prevalence of the regular part of the fundamental solution, which decays to zero in $L^1$ at a suboptimal rate.
Hence, the main difference between the large-time behavior of the heat equation and its Rosenau approximation consists in a slower rate of convergence of the latter towards the Gaussian fundamental solution.
The method of proof also indicates that the rate of convergence in $L^1$ is strictly linked to the regularity of the distribution of the background. For this reason, the argument leading to Theorem \ref{main1} can not be applied to the central difference approximation.
Maybe a different approach will help to clarify if, in some weaker norm, a result similar to that of Theorem \ref{main1} holds also for the central difference type approximation.
\end{remark}

\section{Conclusions}

We studied in this article the validity of the approximation to the linear diffusion equation proposed by Rosenau \cite{Rosenau:1992} as a regularized version of the Chapman-Enskog expansion of hydrodynamics.
This approximation essentially is realized by substituting the heat equation with a linear kinetic equation of Boltzmann type, describing collisions of particles with a fixed background.
This remark allows to consider the Rosenau approximation as a particular realization of a model Boltzmann equation, in which the background distribution is a general probability density with bounded variance.
In addition to the Rosenau distribution, we considered in this article also a point masses background, which furnishes the central difference scheme to solve the linear diffusion equation.
The main differences between the action of the two different backgrounds have been studied in some details. In particular, our analysis put into evidence that the approximation with a regular kernel is a good approximation, {in that it possesses most of the typical properties of the heat equation, including the same large-time behavior, apart from a slower rate of decay towards the fundamental solution and a finite speed of propagation at high Fourier frequencies.}
We were not able to prove an analogous property for the non regular approximation, for which the rate of decay towards the {equilibrium} distribution has been only proven in scaled variables.
The results of the present paper shall lead to new numerical approximations of the whole Rosenau approximation \eqref{Rose}. 
For instance, Wild sums have been successfully used in connection with the approximation of the bilinear Boltzmann equation (see \eg \cite{PR:2001}) to develop a new family of Monte Carlo methods, which possess high order accuracy in time,  and are asymptotic preserving (with respect to $\ve$).
Following the same strategy, our next goal is to develop an efficient Monte Carlo method to compute the Rosenau approximation using our analysis on the linear Wild sums.

\begin{appendix}

\section{Growth in Time of the Sobolev Norm of the Regularized Kernel}

For any given $s<1$, we investigate the growth in time of the quantity
 \[
B_s(t) = e^{-\sigma^2 \, t/\ve^2} \left \| \widehat{\mathcal G_\ve} \left
(\frac\xi{\sqrt{1+t}}\right ) \right \|_{{\dot H^s}}.
 \]
  To avoid inessential heavy notations, in what follows we fix $\ve=\sigma =
1$. In this case
 \begin{equation*}
   B_s(t) = e^{-t} \left[ \int_\RR |\xi|^{2 s} \left\{\exp\left( \frac t{1 + (1+t)^{-1}\xi^2} \right) -1  \right]^2 \, d\xi \right\}^{1/2}.
 \end{equation*}
By a simple change of variable into the integral, we can rewrite $B_s(t)$ in the equivalent form
\be
  \label{me}
  B_s(t) = (1+t)^{s+1/2} \, I_s(t)^{1/2} = (1+t)^{s+1/2}\left\{ \int_\RR |\xi|^{2s} \left[ \exp \left  ( -t \frac{\xi^2}{1+\xi^2}\right ) - \exp(-t)  \right]^2 \, d\xi \right\}^{1/2}.
\ee
The function $I_s(t)$ defined in \fer{me} satisfies the differential equation
 \be\label{def}
 \frac d{dt}I_s(t) = -2 I_s(t) + 2 \int_\RR A_s(\xi,t) \, d\xi,
 \ee
where
 \begin{equation*}
   A_s(\xi,t) = \frac{|\xi|^{2 s}}{1+\xi^2} \exp \left  ( -t \frac{\xi^2}{1+\xi^2}\right ) \left[ \exp \left  ( -t \frac{\xi^2}{1+\xi^2}\right ) - \, \exp(-t)  \right].
 \end{equation*}
For a given positive constant $\delta < 2p+1$, let us set
 \[
 \alpha =  1 - \frac {\delta}{2p +1}
 \]
Then we find
 \be\label{bo1}
 \int_0^{(1+t)^{-\alpha}} A_s(\xi,t)\, d\xi \le \int_0^{(1+t)^{-\alpha}} |\xi|^{2 s}\, d\xi = \frac 1{2 s+1} \frac{(1+t)^\delta}{(1+t)^{2 s+1}}.
 \ee
If now $|\xi| > (1+t)^{-\alpha}$ we have
 \[
\frac{\xi^2}{1+\xi^2} \ge \frac{(1+t)^{-2\alpha}}{1+(1+t)^{-2\alpha}} = \frac 1{1+(1+t)^{2\alpha}} \ge \frac 12 {(1+t)^{-2\alpha}},
 \]
and
\begin{align*}
  \exp \left  ( -t \frac{\xi^2}{1+\xi^2}\right ) -\, \exp(-t) = \frac{ \exp \left  ( -t \frac{\xi^2}{1+\xi^2}\right ) - \, \exp(-t) } {-t \frac{\xi^2}{1+\xi^2} + t} \frac t{1 +\xi^2}  \le \frac t{1 +\xi^2}.
\end{align*}
The previous inequalities imply
\begin{align}
  \int_{(1+t)^{-\alpha}}^{+\infty} A_s(\xi,t)\, d\xi & \le t \exp \left (- t\, (1+t)^{-2\alpha}/2\right ) \int_{(1+t)^{-\alpha}}^{+\infty} \frac{|\xi|^{2 s}}{(1+\xi^2)^2} \, d\xi \notag \\
  & \le c_s \, t \exp \left (- t\, (1+t)^{-2\alpha}/2\right ),  \label{bo2}
\end{align}
where we denoted by $c_s$ the bounded constant
 \[
c_s = \int_\RR \frac{|\xi|^{2 s}}{(1+\xi^2)^2} \, d\xi.
 \]
Putting together inequalities \fer{bo1} and \fer{bo2} we conclude with the upper bound
 \be\label{bo3}
 \bar A_s(t) := \int_{\xi \in \RR} A_s(\xi,t) \, d\xi \le \frac 1{2 s+1} \frac{(1+t)^\delta}{(1+t)^{2 s+1}} +  c_s \,t \exp \left (- t\, (1+t)^{-2\alpha}/2\right ).
 \ee
Hence, considering that the second term in \fer{bo3} decay faster at
infinity than the first one, for any given constant $a>0$ there
exists a bounded constant $C_a$ such that if $t >  a$
 \be\label{bo4}
\bar A_s(t) \le C_a t^{-(2s +1 -\delta)}.
 \ee
Substituting inequality \fer{bo4} into \fer{def}, we obtain, for a given $a>0$, and $t >a$
\begin{align*}
  I_s(t) & \le I_s(a) \, e^{-2t} + e^{-2t} \int_a^t \bar A_s(\tau)\, e^{2 \tau} \, d\tau \\
    & \le I_s(a) \, e^{-2t} + C_a \,  e^{-2t} \int_a^t \tau^{-(2s +1 -\delta)} e^{2 \tau}\, d \tau.
\end{align*}
On the other hand, integration by parts shows that
 \[
 \int_a^t \tau^{-(2s +1 -\delta)} e^{2 \tau}\, d\tau \le \frac {e^{2 t}}2 t^{-(2 s +1 -\delta)} +
 \frac {2 s +1 -\delta}{2a} \int_a^t \tau^{-(2 s +1 -\delta)} e^{2 \tau}\,
 d\tau.
 \]
Choosing $a$ in such a way that
 \[
 \frac {2s +1 -\delta}{2a} \le \frac 12,
  \]
we conclude that, for $t >a$
  \[
    e^{-2t} \int_a^t \tau^{-(2 s +1 -\delta)} e^{2 \tau}\, d \tau \le  t^{-(2 s +1-\delta)},
  \]
  which implies
 \be\label{okk}
I_s(t) \le I_s(a) e^{-2t} +  C_a  t^{-(2s +1-\delta)}.
 \ee
Using \fer{okk} into \fer{me} gives
 \be\label{dec1}
B_s(t) = (1+t)^{s+1/2}I_s(t)^{1/2} \le \sqrt{I_s(a)}(1+t)^{s+1/2} e^{-t} + (1+t)^{\delta/2}\sqrt{C_a}.
 \ee
Since $\delta$ is arbitrarily small, inequality \fer{diff} follows.

\end{appendix}

\section*{Acknowledgment}

  This work was initiated during a visit of Thomas Rey, which thanks the University of Pavia for its grateful hospitality.
  The research of TR was granted by the European Research Council ERC Starting Grant 2009, project 239983-NuSiKiMo.
  GT acknowledges support by MIUR project ``Optimal mass transportation, geometrical and functional inequalities with applications''.

\bibliographystyle{acm}
\bibliography{biblio}
\end{document}